\documentclass[11pt]{article}

\usepackage[utf8]{inputenc} 
\usepackage[T1]{fontenc}    
\usepackage[hypertexnames=false]{hyperref}       
\usepackage{url}            
\usepackage{booktabs}       
\usepackage{amsfonts}       
\usepackage{nicefrac}       
\usepackage{microtype}      
\usepackage{mathtools}
\mathtoolsset{showonlyrefs}
\usepackage{amsthm}

\title{Structural break analysis in high-dimensional
	covariance structure}

\author{
  Valeriy Avanesov \\
  WIAS\\ Mohrenstr. 39\\
  \texttt{avanesov@wias-berlin.de} \\
}

\newtheorem{assumption}{Assumption}
\newtheorem{remark}{Remark}
\newtheorem{lemma}{Lemma}
\newtheorem{theorem}{Theorem}
\newtheorem{definition}{Definition}

\begin{document}

\maketitle

\begin{abstract}
	We consider detection and localization of an abrupt break in the covariance structure of high-dimensional random data. The paper proposes a novel testing procedure for this problem. Due to its nature, the approach requires a properly chosen critical level. In this regard we propose a purely data-driven calibration scheme. The approach can be straightforwardly employed in online setting and is essentially multiscale allowing for a trade-off between sensitivity and change-point localization (in online setting, the delay of detection). The description of the algorithm is followed by a formal theoretical study justifying the proposed calibration scheme under mild assumption and providing guaranties for break detection. All the theoretical results are obtained in a high-dimensional setting (dimensionality $p \gg n$). The results are supported by a simulation study inspired by real-world financial data.
\end{abstract}

\newcommand{\alphacorrected}{\alpha^*}
\newcommand{\Bn}{B_n}
\newcommand{\zetan}{z_n^\eta}
\newcommand{\gammam}{\gamma_-}
\newcommand{\gammap}{\gamma_+}
\newcommand{\gammab}{\gamma^\flat}
\newcommand{\alphaeta}{\alpha_\eta}
\newcommand{\Nn}{\mathcal{N}}
\newcommand{\X}{\mathcal{X}}
\newcommand{\xalphan}{x^\flat_n(\alpha)}
\newcommand{\xalphannn}[2]{x^\flat_{#1}(#2)}
\newcommand{\xalphann}[1]{x^\flat_{#1}(\alpha)}
\newcommand{\R}{\mathbb{R}}
\newcommand{\N}[2]{\mathcal{N}(#1, #2)}
\newcommand{\Nat}{\mathbb{N}}
\newcommand{\E}[1]{\mathbb{E}\left[#1\right]}
\newcommand{\Var}[1]{\mathrm{Var}\left[#1\right]}
\newcommand{\Varboot}[1]{\mathrm{Var}_\flat\left[#1\right]}
\newcommand{\trueSigma}{\Sigma^*}
\newcommand{\hatSigma}{\hat\Sigma}
\newcommand{\trueTheta}{\Theta^*}
\newcommand{\hatTheta}{\hat\Theta}
\newcommand{\Pb}{P^\flat}
\newcommand{\znb}{z^{N^\flat}}
\newcommand{\alphap}{\alpha^+}
\newcommand{\alpham}{\alpha^-}
\newcommand{\Nb}{\mathcal{N}^\flat}
\newcommand{\Zp}{\mathcal{Z}_+}
\newcommand{\Zm}{\mathcal{Z}_-}
\newcommand{\Zb}{{Z}^\flat}
\newcommand{\bZp}{\partial\mathcal{Z}_+}
\newcommand{\bZm}{\partial\mathcal{Z}_-}
\newcommand{\Nr}{\mathcal{N}}
\newcommand{\deltaZ}{\Delta_Z(x)}
\newcommand{\probXBound}{p_s^X(x)}
\newcommand{\Is}{\mathcal{I}_n^{\mathfrak{S}}(t)}
\newcommand{\Ir}{\mathcal{I}^r_n(t)}
\newcommand{\Il}{\mathcal{I}^l_n(t)}
\newcommand{\Istable}{\mathcal{I}_s}
\newcommand{\infnorm}[1]{\left|\left|#1\right|\right|_\infty}
\newcommand{\normtwo}[1]{\left|\left|#1\right|\right|_2}
\newcommand{\inv}[1]{#1^{-1}}
\newcommand{\Ziuv}{Z_{i,uv}}
\newcommand{\hatZiuv}{\hat Z_{i,uv}}
\newcommand{\Zbiuv}{Z_{i,uv}^\flat}
\newcommand{\Zbi}{Z_{i}^\flat}
\newcommand{\Prob}[1]{\mathbb{P} \left\{#1\right\}}
\newcommand{\Probboot}[1]{\mathbb{P^\flat} \left\{#1\right\}}
\newcommand{\abs}[1]{\left|#1\right|}
\newcommand{\ex}[1]{e^{#1}}
\newcommand{\onenorm}[1]{\left|\left|#1\right|\right|_1}
\newcommand{\OneNorm}[1]{\left|\left|\left|#1\right|\right|\right|_1}
\newcommand{\Zbound}[2]{\mathcal{Z}_{#1}(#2)}
\newcommand{\Wbound}[2]{\mathcal{W}_{#1}(#2)}
\newcommand{\Zhatbound}[2]{\hat{\mathcal{Z}}_{#1}(#2)}
\newcommand{\Zboundsqr}[2]{\mathcal{Z}^2_{#1}(#2)}
\newcommand{\Zhatboundsqr}[2]{\mathcal{\hat Z}^2_{#1}(#2)}
\newcommand{\probzbound}[2]{p^{\mathcal{Z}}_{_{#1}}(#2)}
\newcommand{\probwbound}[2]{p^{\mathcal{W}}_{{#1}}(#2)}
\newcommand{\probhatzbound}[2]{p_{\hat{\mathcal{Z}}_{#1}}(#2)}
\newcommand{\probsigmazbound}[3]{p^{\Sigma_Z }_{_{#1}}(#2, #3)}
\newcommand{\probsigmazboundone}{p^{\Sigma_{Z 1}}_{_{s}}(x, q)}
\newcommand{\probsigmazboundtwo}{p^{\Sigma_{Z 2}}_{_{s}}(x)}
\newcommand{\probsigmay}{p^{\Sigma_Y}_{s}(x,q)}
\newcommand{\probSigmaBound}[2]{p^{\Sigma}(#2)}
\newcommand{\probOmegaBound}{p^{\Omega}_s(x,t)}
\newcommand{\Zv}{\overline{Z}}
\newcommand{\Zhatv}{\overline{\hat Z}}
\newcommand{\Zhatvi}{\overline{\hat Z_i}}
\newcommand{\trueSigmaZ}{\Sigma^*_Z}
\newcommand{\hatSigmaZ}{\hat\Sigma_Z}
\newcommand{\trueSigmahatZ}{\Sigma^*_{\hat Z}}
\newcommand{\hatSigmahatZ}{\hat\Sigma_{\hat Z}}
\newcommand{\Wi}{W^{(i)}}
\newcommand{\Sz}{S_Z}
\newcommand{\Sbz}{S_Z^\flat}
\newcommand{\Snbz}{S_Z^{n\flat}}
\newcommand{\Szn}{S_Z^n}
\newcommand{\Sbzn}{S_Z^{n\flat}}
\newcommand{\Znormalized}{Z^{{S}}}
\newcommand{\Zhatnormalized}{\hat{Z}^{{S}}}
\newcommand{\Zbnormalized}{Z^{{S\flat}}}
\newcommand{\maxN}[1]{\max\mathcal{N}_{#1}}
\newcommand{\trueSigmaY}{\Sigma^*_Y}
\newcommand{\hatSigmaY}{\hat\Sigma_Y}
\newcommand{\mymin}[2]{\min_{#1}\left[#2\right]}
\newcommand{\kappaGamma}{\kappa_{\Gamma^*}}
\newcommand{\kappaSigma}{\kappa_{\Sigma^*}}
\newcommand{\kappaTheta}{\kappa_{\Theta^*}}
\newcommand{\Ra}{R_A}
\newcommand{\Rb}{R_B}
\newcommand{\Xb}{X^\flat}
\newcommand{\bootTheta}{\hat\Theta}
\newcommand{\Bb}{B^\flat}
\newcommand{\Bbn}{B^\flat_n}
\newcommand{\patternAb}{\mathcal{A}^\flat}
\newcommand{\patternA}{\mathcal{A}}
\newcommand{\Yb}{Y^\flat}
\newcommand{\empE}[1]{\mathbb{E}_{\mathcal{I}_S}\left[#1\right]}
\newcommand{\empVar}[1]{\mathrm{Var}_{\mathcal{I}_S}\left[#1\right]}
\newcommand{\ione}{i_1}
\newcommand{\itwo}{i_2}
\newcommand{\ithree}{i_3}
\newcommand{\ifour}{i_4}
\newcommand{\Zi}{Z_{i}}
\newcommand{\hatZik}{\hat{Z}_{i,kl}}
\newcommand{\Bernstain}[4]{\frac{\left(#1\right)#2}{3#3}\left(1+ \sqrt{1+\frac{9#3#4}{#2\left(#1\right)^2}} \right)}
\newcommand{\DeltaMomentExp}{\Delta_{M_{exp}}}
\newcommand{\DeltaMomentThree}{\Delta_{M_{3}}}
\newcommand{\DeltaMomentFour}{\Delta_{M_{4}}}
\newcommand{\SigmaZDelta}{\Delta_{\Sigma_Z}}
\newcommand{\xalpha}{x^\flat(\alpha)}
\newcommand{\xm}{x_\alpha^-}
\newcommand{\xp}{x_\alpha^+}
\newcommand{\xmm}{x_\alpha^{--}}
\newcommand{\xpp}{x_\alpha^{++}}
\newcommand{\Rboot}{R_{A^b}}
\newcommand{\RBboot}{R_{B^\flat}}
\newcommand{\Sigmap}{\Sigma_Y^+}
\newcommand{\Sigmam}{\Sigma_Y^-}
\newcommand{\Rsigma}{R_{\Sigma}^{\pm}}
\newcommand{\given}{\vert}
\newcommand{\Tauz}{\mathcal{T}_Z}
\newcommand{\ptauz}{p_{\mathcal{T}_Z}}
\newcommand{\omegaDelta}{\Delta^\Omega_s(x,t)}
\newcommand{\probMoment}[2]{p^M_s(#1)}
\newcommand{\CB}{C_{B}}
\newcommand{\CBb}{C_{B^\flat}}
\newcommand{\hatCBb}{\hat{C}_{B^\flat}}
\newcommand{\deltaY}{\Delta_Y}
\newcommand{\Zvi}{\overline{Z_i}}
\newcommand{\hatZvi}{\overline{\hat Z_i}}
\newcommand{\n}{\mathfrak{N}}
\newcommand{\Tn}{\mathbb{T}_n}
\newcommand{\Tnmax}{\mathbb{T}_{\nmax}}
\newcommand{\nmin}{n_-}
\newcommand{\nmax}{n_+}
\newcommand{\numn}{\abs{\mathfrak{N}}}
\newcommand{\glestimation}{\hatTheta^{GL}}
\newcommand{\mbestimation}{\hatTheta^{MB}}
\newcommand{\tr}{tr}
\newcommand{\maxLambda}[1]{\Lambda\left(#1\right)}
\newcommand{\minLambda}[1]{\lambda\left(#1\right)}
\newcommand{\rank}{rank}
\newcommand{\Tau}{\mathcal{T}}
\newcommand{\xii}{\xi_i}
\newcommand{\hatxii}{\hat\xi_i}
\newcommand{\x}{\mathrm{x}}
\newcommand{\q}{\mathrm{q}}
\newcommand{\twoNorm}[1]{\lvert\lvert#1\rvert\rvert_{2}}
\newcommand{\probxbound}[2]{p_{X_#1}(#2)}
\newcommand{\Asc}{A^{sc}(s)}
\newcommand{\Hnull}{\mathbb{H}_0}
\newcommand{\Halt}{\mathbb{H}_1}
\newcommand{\Rt}{R_{\hat T}}
\newcommand{\s}{\mathfrak{S}}
\newcommand{\Zcb}{\mathcal{Z}^\flat}
\newcommand{\zb}{z^\flat}
\newcommand{\zm}{z_-}
\newcommand{\zp}{z_+}
\newcommand{\zmn}{z_-^0}
\newcommand{\Zbn}{\mathcal{Z}^\flat(a)}
\newcommand{\zpn}{z_+^0}
\newcommand{\trueOmega}{\Omega^*}
\newcommand{\hatOmega}{\hat\Omega}
\newcommand{\deltaOmega}[3]{\Delta^\Omega_{#1}(#2, #3)}
\newcommand{\deltaOmegaProb}[3]{p^\Omega_{#1}(#2, #3)}
\newcommand{\SigmaOne}{\Sigma_1}
\newcommand{\SigmaTwo}{\Sigma_2}

\section{Introduction}

The analysis of high dimensional time series is crucial for many fields including neuroimaging and financial engineering.
There one often has  to deal with processes involving abrupt structural breaks which necessitates a corresponding adaptation of the model and/or the strategy.
Structural break analysis comprises determining if an abrupt change is present  in the given sample and if so,
estimating the change-point, namely the moment in time when it takes place. In literature
both problems may be referred to as {\it change-point} or {\it break detection}. In this study we will
be using terms {\it break detection} and {\it change-point localization} respectively in order to distinguish between them.
The majority of approaches consider only a univariate process \cite{limitTheoremsCPAnalysis, aue2013}.
However, in recent years the interest for multi-dimensional approaches has increased. Most of them cover the case of fixed dimension \cite{Matteson2015, Lavielle2006, aue2009, xie2013, zou2014}. Some approaches \cite{cho2016, jirak2015, Cho2015} feature {\it high-dimensional} theoretical guaranties but only the case of dimensionality polynomially growing in sample size is covered. The case of exponential growth has not been considered so far.

In order to detect a break, a test statistic is usually computed for each point $t$ (e.g. \cite{Matteson2015}). The break is detected if the maximum of these values exceeds a certain \textit{threshold}. A proper choice of the latter may be a tricky issue.  The classical approach to the problem is based on the asymptotic behavior of the statistic \cite{limitTheoremsCPAnalysis, aue2013, aue2009, jirak2015, biau2016, zou2014}. As an alternative, permutation \cite{jirak2015, Matteson2015} or parametric bootstrap may be used \cite{jirak2015}. Clearly, it seems attractive to choose the threshold in a solely data-driven way employing bootstrap as it is suggested in the recent paper by \cite{cho2016}, but a rigorous bootstrap validation is still an open question, which we address in the study.

In the current study we are interested in a particular kind of a break -- an abrupt transformation in the covariance matrix -- which is motivated by applications to finance and neuroimaging. In finance the dynamics of the covariance structure of a high-dimensional process modeling return rates is crucial for a proper asset allocation in a portfolio \cite{Serban2007,Bauwens2006,Engle1990,Mikosch2009}. Analogously, break analysis in covariance structure of data in functional Magnetic Resonance Imaging is  particularly important for the research on neural diseases as well as in context of brain development with emphasis on characterization of the re-configuration of the brain during learning  \cite{Bassett_Wymbs_Porter_Mucha_Carlson_Grafton_2010, sporns2011, journals/brain/Friston11}.

One approach allowing for the change-point  localization is developed in \cite{Lavielle2006}, the corresponding significance testing problem is considered in \cite{aue2009}. However, neither of these papers addresses the high-dimensional case.

A widely used break detection approach (named CUSUM) \cite{Cho2015, aue2009, jirak2015} suggests to compute a statistic at a point $t$ as a distance of estimators of some parameter of the underlying distributions obtained using all the data before and after that point. This technique requires the whole sample to be known in advance, which prevents it from being used in \textit{online} setting. In order to overcome this drawback we propose the following augmentation: choose a window size $n \in \Nat$ and compute parameter estimators using only $n$ points before and $n$ points after the \textit{central point} $t$ (see Section \ref{apprsec} for formal definition). Window size $n$ is an important parameter and its choice is case-specific (see Section \ref{secsens} for theoretical treatment of this issue). Using small window results in high variability and low sensitivity, while large window implies higher uncertainty in change-point localization yielding the issue of a proper choice of window size. The \textit{multiscale} nature of the proposed method enables us to 	incorporate the advantages of narrower and wider windows by considering multiple window sizes at once in order for wider windows to provide higher sensitivity while narrower ones improve change-point localization.

The contribution of our study is the development of a novel  structural break analysis approach which is
\begin{itemize}
	\item high-dimensional, allowing for up to exponential growth of the dimensionality with the window size
	\item suitable for online setting
	\item multiscale, attaining trade-off between break detection sensitivity and change-point localization accuracy
	\item using a fully data-driven calibration scheme rigorously justified under mild assumptions
	\item featuring formal sensitivity guaranties in high-dimensional setting
\end{itemize}

We consider the following setup. Let $X_1,...,X_N\in\R^p$ denote a sample of independent zero-mean vectors. In online setting the sample size is not fixed in advance. The goal is to test the hypothesis

\begin{equation}
\Hnull \coloneqq \{\forall i : \Var{X_i} = \Var{X_{i+1}}   \}
\end{equation}
versus the alternative suggesting the existence of a break:

\begin{equation}
\Halt  \coloneqq  \left\{ \exists \tau : \Var{X_{\tau}} \neq  \Var{X_{\tau+1}}  \right\}
\end{equation}
and localize the change-point $\tau$ as precisely as possible or (in online setting) to detect a break as soon as possible.

To this end we define a family of test statistics in Section \ref{apprsec} which is followed by Section \ref{secBOoot} describing a data-driven (bootstrap) calibration scheme. Section \ref{secbv} presents and discusses a theoretical result justifying the bootstrap scheme while Section \ref{secsens} presents a sensitivity result providing a lower bound for a window size $n$ necessary to detect a break of a given extent and hence bounding the uncertainty of the change-point localization (or the delay of detection in online setting). Finally, Section \ref{Simulation} presents a simulation study inspired by real-world financial data supporting the theoretical findings and demonstrating superiority of our approach to a recent one.

\section{Proposed approach}
The first part of this Section formally defines the test statistics while the second part concentrates on the calibration scheme. Informally, the test statistics may be defined as follows. Provided that the break may happen only at point $t$, one could estimate the covariance matrix using $n$ data-points to the left of $t$, estimate it again using $n$ data-points to the right of it and use the norm of their difference as a test statistic $B_n(t)$. Yet, in practice one does not usually possess such knowledge, therefore we propose to maximize these statistics over all possible locations $t$ yielding $B_n$. Finally, in order to attain a trade-off between break detection sensitivity and change-point localization accuracy we propose a multiscale approach considering multiple window sizes $n \in \n$ and multiple respective test statistics $\{B_n\}_{n \in \n}$ at once.
\subsection{Definition of the test statistics}\label{apprsec}

Now we present a formal definition of the test statistic.
In order to detect a break we consider a set of window sizes $\n \subset \mathbb{N}$. Denote the size of the widest window as $\nmax$ and of the narrowest as $\nmin$. Given a sample of length $N$ for each window size $n \in \n$ define a set of central points $\Tn \coloneqq \{n+1, n+2,... N-n+1\}$.  Next, for all $n \in \n$ define a set of indices which belong to the window on the left side from the central point $t \in \Tn$ as $\Il \coloneqq \{t-n,t-n+1, ... , t-1\}$ and correspondingly for the window on the right side define $\Ir \coloneqq \{t, t+1,... , t+n-1\}$.
Denote the sum of numbers of central points for all window sizes $n \in \n$ as

\begin{equation}\label{Tdef}
T \coloneqq \sum_{n \in \n} \abs{\Tn}.
\end{equation}
For each window size $n \in \n$ and each central point $t \in \Tn$ define a pair of estimators of covariance matrix as

\begin{equation}
\hatSigma^l_n(t) \coloneqq \frac{1}{n} \sum_{i \in \Il} X_i X_i^T \text{   and     }\hatSigma^r_n(t) \coloneqq \frac{1}{n} \sum_{i \in \Ir} X_i X_i^T.
\end{equation}

Let some subset of indices $\Istable \subseteq 1..N$ of size $s$ (possibly, $s = N$) be chosen. Define a scaling diagonal matrix $$S = diag(\sigma_{1,1},\sigma_{1,2}... \sigma_{p,p-1}, \sigma_{p,p})$$  where the elements $\sigma_{j,k}$ are standard deviations of corresponding elements of $X_i X_i^T$ averaged over $\Istable$:
\begin{equation}
\sigma_{j,k}^2 \coloneqq \frac{1}{s} \sum_{i \in \Istable}\Var{(X_i X_i^T)_{jk}}.
\end{equation}
In practice the matrix $S$ is usually unknown, hence we propose to plug-in empirical estimators $\hat{\sigma}_{j,k}$.

For each window size $n \in \n$ and central point $t \in \Tn$ we define a test statistic $B_n(t)$

\begin{equation}  \label{Adef}
\begin{split}
B_n(t) &\coloneqq \infnorm{\sqrt{\frac{n}{2}}\inv{S}\overline{(\hatSigma^l_n(t) - \hatSigma^r_n (t))}}.
\end{split}
\end{equation}
Here and below we write $\overline{A}$ for a vector composed of stacked columns of matrix $A$ and use $\infnorm{\cdot}$ to denote the sup norm.
Finally, the family of test statistics $\{B_n\}_{n \in \n}$ is obtained via maximization over the central points:
\begin{equation}
\Bn \coloneqq \max_{t \in \Tn} \Bn(t).
\end{equation}

\begin{remark}
	Generally,  one can choose the diagonal matrix $S$ arbitrarily as long as its elements are bounded. The choice does not affect Theorems \ref{theTheorem} and \ref{senst}. However, we prefer to bring all the elements of the covariance matrices to the same scale first, so the test focuses on a relative change. Ideally, we would like to use the $\sigma_{j,k}^2$, yet due to its unavailability we resort to their empirical estimates, whose consistency can be easily demonstrated based on Assumption~\ref{subGaussianVector}.
\end{remark}

\subsection{Decision rule and bootstrap calibration scheme} \label{secBOoot}

Our approach rejects $\Hnull$ in favor of $\Halt$ if at least one of statistics $B_n$ exceeds a corresponding threshold $\xalphan$ or formally if $\exists n \in \n : \Bn > \xalphan$.

In order to choose thresholds $\xalphan$ the following bootstrap scheme is proposed. Define vectors $\hat Z_i$ for $i \in \Istable$ as
\begin{equation}
\hat Z_i \coloneqq \overline{X_i X_i^T - \frac{1}{s} \sum_{i \in \Istable} X_i X_i^T}.
\end{equation}
Elements $\Zbi$ for $i \in 1..N$ of bootstrap sample are proposed to be drawn with replacement from the set ~$\bigcup_{i \in \Istable} \{\hat \Zi, -\hat \Zi\}$. Denote the measure which $\Zbi$ are distributed with respect to as $\mathbb{P}^\flat$. By construction $\mathbb{P}^\flat$ is not absolute continuous w.r.t to Lebesgue measure, which is not a problem per se, yet ``high jumps'' naturally complicate quantile estimation. Bringing in both $\hat \Zi$ and $-\hat \Zi$ reduces the ``jumps''.

Now we are ready to define a bootstrap counterpart $\Bb_n(t)$ of $B_n(t)$ for all $n \in \n$ and $t \in \Tn$ as
\begin{equation} \label{bootAnDef}
\Bb_n(t) \coloneqq \infnorm{\sqrt{\frac{n}{2}}\inv{S} \left( \frac{1}{n}\sum_{i \in \Il} \Zbi  - \frac{1}{n}\sum_{i \in \Ir} \Zbi \right)}.
\end{equation}
The counterparts $\Bbn$ of $\Bn$  for all $n \in \n$ are naturally defined as
\begin{equation}
\Bbn \coloneqq \max_{t \in \Tn} \Bbn(t).
\end{equation}
Now for each given $\x \in (0,1)$ we can define quantile functions $z^\flat_n(\x)$ such that
\begin{equation}\label{tailfunctiondef}
z^\flat_n(\x) \coloneqq \inf \left\{z :  \Probboot{\Bbn > z} \le\x\right\} .
\end{equation}
Next for a given significance level $\alpha$ we apply multiplicity correction choosing $\alphacorrected$ as
\begin{equation}\label{alphastardef}
\alphacorrected \coloneqq \sup \left\{\x : \Probboot{\exists n \in \n : \Bbn > z^\flat_n(\x)}  \le \alpha\right\}
\end{equation}
and finally choose thresholds as $\xalphan \coloneqq z^\flat_n(\alphacorrected)$.

\begin{remark}\label{remark1}
	In most of the cases one may simply choose $\Istable = 1...N$ but at the same time it seems appealing to use some sub-sample which a priori does not include a break, if such information is available. On the other hand, the bootstrap justification result (Theorem \ref{theTheorem}) and sensitivity result (Theorem \ref{senst}) benefit from larger set $\Istable$. The experimental comparison of these options is given in Section \ref{Simulation}.
\end{remark}

\subsection{Change-point localization}

In order to localize a change-point we have to assume that $\Istable \subseteq 1..\tau$. Consider the narrowest window detecting a change-point as $\hat n$:
\begin{equation}
\hat n \coloneqq \min \left\{n \in \n : \Bn > \xalphan \right\}
\end{equation}
and the central point where this window detects a break for the first time as 
\begin{equation}
\hat \tau \coloneqq \min \left\{ t \in \mathbb{T}_{\hat n} : B_{\hat n}(t) > x^\flat_{\hat n}(\alpha) \right\}.
\end{equation}
By construction of the family of the test statistics we conclude (up to the confidence level $\alpha$) that the change-point $\tau$ is localized in the interval

\begin{equation}
\left[\hat \tau - \hat n ; \hat \tau + \hat n - 1\right].
\end{equation}
Clearly, if a non-multiscale version of the approach is employed, i.e. $\abs{\n} = \{n\}$, $n = \hat n$
and precision of localization (delay of the detection in online setting) equals $n$.

\section{Bootstrap validity result}\label{secbv}
This section states and discusses the theoretical result demonstrating validity of the proposed bootstrap scheme i.e.

\begin{equation}\label{sort}
\Prob{\forall n \in \n : \Bn \le \xalphan} \approx 1-\alpha.
\end{equation}

Our theoretical results require the tails of the underlying distributions to be light. Specifically, we impose Sub-Gaussianity vector condition.

\begin{assumption}
	\label{subGaussianVector}
	\begin{equation}
	\exists L >0 : \forall i\in 1..N  \sup_{\substack{a \in \R^p \\ \normtwo{a} \le 1 }}\E{\exp{\left(\left(\frac{a^T	X_i}{L}\right)^2\right)}} \le 2.
	\end{equation}
\end{assumption}

\begin{theorem} \label{theTheorem}
	Let Assumption \ref{subGaussianVector} hold and let $X_1, X_2,...,X_N$ be i.i.d. Allow the parameters $p, \abs{\n}, s, \nmin, \nmax$ grow with $N$. Further let $N > 2\nmax \ge 2\nmin $ and $N > s$ and let the minimal window size $\nmin$ and the size $s$ of the set $\Istable$ grow fast enough

	\begin{equation}\label{slowTerm}
	\frac{\abs{\n}L^4 \log^{8}(pN)}{\min\{\nmin, s\}} = o(1).
	\end{equation}
	Then
	\begin{equation}
	\abs{\Prob{\forall n \in \n : \Bn \le \xalphan} - (1-\alpha)}  = o_P(1),
	\end{equation}
\end{theorem}

The formal proof of the theorem can be found in Supplementary Materials Section \ref{secProof} along with the finite-sample-size version of the result.

\begin{remark}
	The form of assumption \eqref{slowTerm} is mostly driven by the remainder term in the Gaussian Approximation Result (Lemma \ref{generalGAR}). Optimality of the term is discussed in \cite{Chernozhukov2014}. The authors conjecture that it is minimax optimal up to the power the logarithm is raised to ($\log^7(pn)$). As the case of isotropic vectors $X_i$ demonstrates, the power cannot be less than $3$.
	Therefore, assumption \eqref{slowTerm} surely may not be weaker (in terms of dependence on $p, N, \nmin, \text{ and } s$) than
	\begin{equation}
		\log^3(pN) \ll \min\{\nmin, s\},
	\end{equation}
	as long as we use a Gaussian Approximation Result, which is the mainstream approach to prove a bootstrap validity result. Recently a successful attempt \cite{Deng} was made to bypass Gaussian approximation, which brought the power down from $7$ to $5$. Hence, one can hypothesize that Theorem \ref{theTheorem} can not be re-established under a condition weaker than
	\begin{equation}
		\log^5(pN) \ll \min\{\nmin, s\}.
	\end{equation}
	Thus we conjecture the assumption \eqref{slowTerm} is nearly optimal.

	At the same time the good performance of the approach exhibited in the simulation study (see Section \ref{Simulation}) suggests there is a possibility to obtain a significantly better theoretical results for a narrower distribution family. We leave suggestion of such a family and further investigation for the future research.
\end{remark}

\paragraph{Proof discussion}
The proof of the bootstrap validity result mostly relies on the high-dimensional central limit theorems obtained by \cite{Chernozhukov2014}. That paper also presents bootstrap justification results, yet does not include a comprehensive bootstrap validity statement. The theoretical treatment is complicated by the randomness of $\xalphan$. Indeed, consider Lemma \ref{tvlemma}. One cannot trivially obtain result of sort \eqref{sort} substituting $\{\xalphan\}_{n \in \n}$ in \eqref{eqtv} due to the randomness of $\xalphan$ and dependence between $\xalphan$ and $B_n$. We overcome this by means of so-called ``sandwiching'' proof technique (see Lemma \ref{sandwitch}), initially used by \cite{SpokWillrich} and extended by \cite{buzun}. The authors of \cite{SpokWillrich} had to assume normality and low dimensionality of the data, while in \cite{buzun} only continuous probability measures $\mathbb{P}$ and $\mathbb{P^\flat}$ were considered. Our result is free of such limitations.

\paragraph{Online setting}\label{online}
As one can easily see, the theoretical result is stated in off-line setting, when the whole sample of size $N$ is acquired in advance. In online setting we suggest to control the probability $\alpha$  to raise a false alarm for at least one central point $t$ among $N$ data points (which differs from classical techniques controlling the mean distance between false alarms \cite{aries2007optimal}). Having $\alpha$ and $N$ chosen, one should acquire $s$ data-points (set $\{X_i\}_{i \in \Istable}$) and employ the proposed bootstrap scheme with the bootstrap samples of length $N$ in order to obtain the critical values. Next, the approach can be naturally applied in online setting and Theorem \ref{theTheorem} guarantees the capability of the proposed bootstrap scheme to control the aforementioned probability to raise a false alarm.

\section{Sensitivity result}\label{secsens}

Consider the following setting. Let there be index $\tau$, such that $\{X_i\}_{i \le \tau}$ are i.i.d. and $\{X_i\}_{i > \tau}$ are i.i.d. as well. Denote covariance matrices $\SigmaOne \coloneqq \E{X_1X_1^T}$ and $\SigmaTwo \coloneqq \E{X_{\tau+1}X_{\tau+1}^T}$.   Define the break extent $\Delta$ as
\begin{equation}\label{breakextentdef}
\Delta \coloneqq \infnorm{{\SigmaOne - \SigmaTwo}}.
\end{equation}
 The question is, how large the window size $\nmax$ should be in order to reliably reject $\Hnull$.

\begin{theorem} \label{senst}
	Let Assumption \ref{subGaussianVector} hold and let $X_1, X_2,...,X_N$ be i.i.d. Allow the parameters $p, \abs{\n}, s, \nmin, \nmax$ grow with $N$ and let the break extent $\Delta$ decay with $N$. Further let $N > 2\nmax \ge 2\nmin $, $N > s$, $\Istable \subset 1..\tau$ and let the minimal window size $\nmin$, the size $s$ of the set $\Istable$ and the maximal window size $\nmax$ grow fast enough

	\begin{equation}
	\frac{\abs{\n}L^4 \log^{8}(pN)}{\min\{\nmin, s\}} = o(1),
	\end{equation}

	\begin{equation}\label{nmaxbound}
	\frac{\log (pN)}{\nmax \Delta^2} = o(1).
	\end{equation}
	Then $\Hnull$ will be rejected with probability approaching 1.
\end{theorem}

The formal proof along with the finite-sample-size version is given in Supplementary Materials Section \ref{sensResProof}.

\paragraph{Discussion of sensitivity result}
The assumption $\Istable \subseteq 1..\tau$ is only technical. A similar result may be proven without relying on it by methodologically the same argument. Really, if the assumption is violated, the method is calibrated for a matrix $\eta\Sigma_1 + (1-\eta)\Sigma_2$, where $\eta\in[0,1]$ and depends on the portion of time-points before and after the break included in the $\Istable$. Clearly (see the proof for details), this changes the bound for critical level only by some multiplicative constant, which does not affect the asymptotic result in question.

Obviously, we still cannot explicitly compute the window size sufficient for reliable break detection, since it depends on the underlying distributions. However this result guarantees that the sensitivity of the test does not vanish in high-dimensional setting.

\paragraph{Online setting}
Theorem \ref{senst} is established in offline setting as well. In online setting it guarantees that the proposed approach can reliably detect a break of an extent not less than $\Delta$ with a delay at most $\nmax$ satisfying \eqref{nmaxbound}.

\paragraph{Change-point localization guaranties}
Theorem \ref{senst} implies by construction of statistic $\Bn$ that the change-point can be localized with precision up to $\nmax$. Hence the bound \eqref{nmaxbound} provides the bound for change-point localization accuracy.

\section{Simulation study} \label{Simulation}

\subsection{Real-world covariance matrices}
 We have downloaded stock market quotes for $p=87$ companies included in S\&P $100$ with $1$-minute intervals for approximately a week ($N=2211$) using the API provided by Google Finance\footnote{\url{https://www.google.com/finance}}.  A sample of interest was composed of $1$-minute log returns for each of the companies. Our approach with window size $\n =\{30\}$ has detected a break at confidence level $\alpha = 0.05$, while the approach proposed by \cite{Matteson2015} (referred to as ecp below) has detected nothing. The change-point was localized at the morning of Monday 19 December 2016 (the day when the Electoral College had voted).

Discarding the portion of the data around the estimated change-point we have acquired a pair of data samples which both approaches fail to detect a break in. Denote the realistic covariance matrices estimated on each of these samples as $\Sigma_1$ and $\Sigma_2$. These matrices are publicly available \footnote{\url{fill.me/data.zip}}.

Our implementation is available\footnote{\url{https://github.com/akopich/covcp}} under GPLv2.

\subsection{Design of the simulation study, results and discussion}

The goal of the current simulation study is to verify that the bootstrap procedure controls first type error rate and evaluate the power of the test and compare it to the power of ecp. Hence we need to generate two types of realistic datasets -- with and without a break for power and first type error rate estimation respectively. In order to generate a dataset without a break we independently draw $520$ vectors from normal distribution $\N{0}{\Sigma_1}$. As for the datasets including a break, they are generated by binding $400$ vectors independently drawn from  $\N{0}{\Sigma_1}$ and $120$ vectors independently drawn from $\N{0}{\Sigma_2}$. Clearly, the data generated in such a fashion satisfies Assumption \ref{subGaussianVector} as a Gaussian vector is also sub-Gaussian.

The results obtained in the simulation study are given in Table \ref{restabel}. One can easily see that the proposed test exhibits proper control  of the first type error rate. In fact, it is conservative due to $\le$  signs entering the definitions of the tail-functions \eqref{tailfunctiondef} and corrected significance level $\alpha^*$ \eqref{alphastardef}. The issue may be mitigated by drawing more bootstrap samples during the calibartion stage, yet we leave this out of the scope as our theoretical results effectively presume availability of an infinite number of bootstrap samples. ecp (being tested in the same setting) has demonstrated proper first type error rate as well, but the power did not exceed $0.1$. So, our approach outperforms ecp in all cases apart from $\n=\{7\}$ and $\Istable = 1..100$.

As expected, the power is higher for larger windows and it may be decreased by adding narrower windows into consideration which is the price to be paid for better change-point localization.

It should be noted that contrary to the intuition expressed in Remark \ref{remark1} using only a data sub-sample which a priori does not include a break does not necessarily improve the power of the test.

For the case of $\Istable = 1..100 \subset 1..\tau$ Table \ref{restabel} also provides mean precision of change-point localization. One can see, that multiscale approach significantly improves it.

\begin{table}[]
    \centering
\caption{First type error rate and power exhibited by the proposed approach for various choice of set of window sizes $\n$ and sub-set used for bootstrap $\Istable$ at significance level $\alpha = 0.05$. For the case $\Istable \subset 1..\tau$ mean precision of change-point localization is reported as well.}
\label{restabel}
\begin{tabular}{|c|c|c|c|c|c|}
\hline
                 & \multicolumn{2}{c|}{$\Istable = 1..520$}                             & \multicolumn{3}{c|}{$\Istable = 1..100$}                                            \\ \hline
$\n$             & \begin{tabular}[c]{@{}c@{}}I type \\ error rate\end{tabular} & power & \begin{tabular}[c]{@{}c@{}}I type \\ error rate\end{tabular} & power & localization \\ \hline
$\{60\}$         & .02                                                          & 1.00  & .00                                                          & .90   & 60           \\ \hline
$\{30\}$         & .01                                                          & .90   & .00                                                          & .52   & 30           \\ \hline
$\{15\}$         & .00                                                          & .76   & .00                                                          & .38   & 15           \\ \hline
$\{7\}$          & .00                                                          & .34   & .00                                                          & .03   & 7            \\ \hline
$\{60,30\}$      & .01                                                          & .99   & .00                                                          & .84   & 47.1         \\ \hline
$\{60,30,15\}$   & .01                                                          & .99   & .00                                                          & .82   & 41.1         \\ \hline
$\{60,30,15,7\}$ & .01                                                          & .99   & .00                                                          & .78   & 42.0         \\ \hline
$\{30,15\}$      & .01                                                          & .90   & .00                                                          & .49   & 21.8         \\ \hline
$\{30,15,7\}$    & .01                                                          & .84   & .00                                                          & .34   & 19.9         \\ \hline
\end{tabular}
\end{table}

\section*{Acknowledgments}
The research of ``Project Approximative Bayesian inference and model selection for stochastic differential equations (SDEs)'' has been partially funded by Deutsche Forschungsgemeinschaft (DFG) through grant CRC 1294 ``Data Assimilation'', ``Project Approximative Bayesian inference and model selection for stochastic differential equations (SDEs)''.

We thank Vladimir Spokoiny, Karsten Tabelow, and the anonymous reviewer for comments and discussions that greatly improved the manuscript.

\bibliographystyle{plain}
\bibliography{main}

\newpage

\setcounter{equation}{0}
\setcounter{figure}{0}
\setcounter{table}{0}
\setcounter{section}{0}
\setcounter{theorem}{0}
\setcounter{lemma}{0}
\setcounter{page}{1}

\section{Notation}
The proof make use of numerous notation. For convinience of the reader we provide the Table \ref{cheat}.

\begin{table}[]
	\begin{tabular}{|l|l|l|}
		\hline
		Notation                    & Definition                                                                                     & Meaning                                                                                                             \\ \hline
		$Y_{\cdot i}$               & see \eqref{BFMTogether}                                                                        & \begin{tabular}[c]{@{}l@{}}Gaussian vectors involved in \\ Gaussian Approximation of $B_n$\end{tabular}             \\ \hline
		$Y_{\cdot i}^{\flat}$       & see \eqref{BFMbootstrapTogether}                                                               & \begin{tabular}[c]{@{}l@{}}Gaussian vectors involved in \\ Gaussian Approximation of $B_n^{\flat}$\end{tabular}     \\ \hline
		$\trueSigmaY$               & see \eqref{truesigmaYDef}                                                                      & \begin{tabular}[c]{@{}l@{}}Average covariance matrix of \\ vectors $Y_{\cdot i}$\end{tabular}                       \\ \hline
		$\hatSigmaY$                & see \eqref{hatsigmaYDef}                                                                       & \begin{tabular}[c]{@{}l@{}}Average covariance matrix of vectors \\  $Y_{\cdot i}^{\flat}$\end{tabular}              \\ \hline
		$\deltaY$                   & see Lemma \ref{trueHatSigma}                                                                   & Bound for $\infnorm{\hatSigmaY- \trueSigmaY}$                                                                       \\ \hline
		$\Rb$                       & \begin{tabular}[c]{@{}l@{}}see Lemma\ref{SGar}\end{tabular}                                    & \begin{tabular}[c]{@{}l@{}}Residual term in Gaussian \\ approximation of $B_n$\end{tabular}                         \\ \hline
		$\RBboot$                   & see Lemma \ref{SbGar}                                                                          & \begin{tabular}[c]{@{}l@{}}Residual term in Gaussian \\ approximation of $B_n^{\flat}$\end{tabular}                 \\ \hline
		$\Zbound{s}{\kappa}$        & see Lemma \ref{zboundlemmaAll}                                                                 & \begin{tabular}[c]{@{}l@{}}Uniform probabilistic \\ bound for sup-norms of $Z_i$\end{tabular}                       \\ \hline
		$\probzbound{s}{\kappa}$    & $s\ex{-\kappa}$                                                                                & \begin{tabular}[c]{@{}l@{}}The probability for the bound \\ $\Zbound{s}{\kappa}$ to be exceeded\end{tabular}        \\ \hline
		$W_i$                       & $\overline{X_i X_i^T}$                                                                         & \begin{tabular}[c]{@{}l@{}}A vectorized summand involved in \\ the definition of an empirical covariance\end{tabular} \\ \hline
		$\trueOmega$                & $\E{\left(W_1 - \overline{\trueSigma}\right) \left(W_1 - \overline{\trueSigma}\right)^T}$      & Covariance of $W_i$                                                                                                 \\ \hline
		$\hatOmega$                 & $\empE{{\left(W_i - \overline{\trueSigma}\right) \left(W_i - \overline{\trueSigma}\right)^T}}$ & Empirical covariance of $W_i$ w.r.t to $\Istable$                                                                   \\ \hline
		$\delta_s(\chi)$            & see Lemma \ref{sigmaconcentration}                                                             & \begin{tabular}[c]{@{}l@{}}Probabilistic bound for \\ $\infnorm{\overline{\trueSigma}  - \empE{W_i}}$\end{tabular}  \\ \hline
		$\probSigmaBound{s}{\chi}$  & $2\ex{-\chi}$                                                                                  & \begin{tabular}[c]{@{}l@{}}Probability for the bound \\ $\delta_s(\chi)$ to be exceeded\end{tabular}                \\ \hline
		$\Wbound{s}{\x}$            & $\x^2 + \infnorm{\trueSigma}$                                                                  & \begin{tabular}[c]{@{}l@{}}Uniform probabilistic bound for \\ $\infnorm{W_i - \overline{\trueSigma}}$\end{tabular}  \\ \hline
		$\probwbound{s}{\x}$        & $ps\ex{-\x}$                                                                                   & \begin{tabular}[c]{@{}l@{}}Probability for the bound \\ $\Wbound{s}{\x}$ to be exceeded\end{tabular}                \\ \hline
		$\deltaOmega{s}{t}{\x}$     & see Lemma \ref{omegaConcentation}                                                              & Probabilistic bound for $\infnorm{\trueOmega - \hatOmega}$                                                          \\ \hline
		$\deltaOmegaProb{s}{t}{\x}$ & $p^2e^{-t} + \probwbound{s}{\x}$                                                               & \begin{tabular}[c]{@{}l@{}}Probability for the bound \\ $\deltaOmega{s}{t}{\x}$ to be exceeded\end{tabular}         \\ \hline
	\end{tabular}\label{cheat}\caption{Proof notation}
\end{table}

\section{Proof of bootstrap validity result} \label{secProof}

\begin{theorem} \label{theTheoremFin}
	Let Assumption \ref{subGaussianVector} hold and let $X_1, X_2,...,X_N$ be i.i.d. Moreover, assume that the residual $R < \alpha/2$
	where

	\begin{equation}
	R\coloneqq \left(3+2\abs{\n}\right)  \left(2\Rb + 2\RBboot + \Rsigma\right),
	\end{equation}

	\begin{equation}
	\Rsigma \coloneqq C\deltaY^{1/3} \log^{2/3}\left(Tp^2\right),
	\end{equation}
	$\deltaY$,  $\Rb$ and $\RBboot$ are defined in Lemmas \ref{trueHatSigma}, \ref{SGar} and \ref{SbGar} respectively and $C$ is an independent positive constant.
	Then for all positive $\x$, $t$ and $\chi$
	it holds that
	\begin{equation}
	\abs{\Prob{\forall n \in \n : \Bn \le \xalphan} - (1-\alpha)}  \le R + 2(1-q),
	\end{equation}
	where \begin{equation} \label{probsucc}
	q\coloneqq1 - \probzbound{s}{\kappa} - \deltaOmegaProb{s}{t}{\x} - \probwbound{s}{\x} - \probSigmaBound{s}{\chi},
	\end{equation}
	probabilities $\probzbound{s}{\kappa}$, $\deltaOmegaProb{s}{t}{\x}$, $\probwbound{s}{\x}$ and $\probSigmaBound{s}{\chi}$  come from Lemmas \ref{zboundlemmaAll}, \ref{omegaConcentation}, \ref{wboundlemma} and \ref{sigmaconcentration} respectively
	and quantiles $\{\xalphan\}_{n \in \n}$ are yielded by bootstrap procedure described in Section \ref{secBOoot}.
\end{theorem}

\paragraph{Proof sketch}
The proof consists of four straightforward steps.
\begin{enumerate}
	\item Approximate statistics $\Bn$ by norms of a high-dimensional Gaussian vector up to the residual $\Rb$ using the high dimensional central limits theorem by \cite{Chernozhukov2014}.
	\item Similarly, we approximate bootstrap counterparts $\Bbn$ of the statistics up to the residual $\RBboot$.
	\item Prove that the covariance matrix of the Gaussian vector used to approximate $\Bbn$ in step 2 is concentrated in the ball of radius $\deltaY$ centered at its real-world counterpart involved in step 1 and employ the Gaussian comparison result provided by \cite{Chernozhukov2014} and \cite{chernozhukov2013}.
	\item Finally, obtain the bootstrap validity result combining the results of steps 1-3.
\end{enumerate}

\begin{proof}
	Proof of the Theorem consists in applying Lemmas \ref{SGar}, \ref{SbGar} and \ref{tvlemma} justifying applicability of sandwiching Lemma \ref{sandwitch} on a set of probability not less than $q$ (defined by \eqref{probsucc}) which are followed by applying Lemma \ref{conditioning}.
\end{proof}

\paragraph{Finite-sample-size bootstrap validity result discussion}

The remainder terms $\Rb$, $\RBboot$ and $\Rsigma$ involved in the statement of Theorem \ref{theTheorem} are rather complicated. Here we just note that for $p$, $s$, $N$, $\nmin$, $\nmax \rightarrow +\infty$, $N > 2\nmax$, $\nmax \ge \nmin$

\begin{equation}
\Rb \le C_1\left(\frac{L^4\log^{7}\left(p^2T\nmax \right)}{\nmin}\right)^{1/6},
\end{equation}

\begin{equation}\label{rbbootdisc}
\RBboot\le C_2\left(\frac{L^4\log^{7}\left(p^2T\nmax \right)}{\nmin}\right)^{1/6}\log(ps),
\end{equation}

\begin{equation}
\Rsigma \le C_3\left(\frac{L^4\log^4(ps)}{s}\right)^{1/6} \log^{2/3} \left(p^2T\right),
\end{equation}
while the parameters $\kappa,\x, \chi, t$ are chosen in order to ensure the probability $q$ defined by \eqref{probsucc} to be above $0.995$, e.g.

\begin{equation}\label{xdef}
\x = 7.61 + \log (ps),
\end{equation}

\begin{equation}\label{kappadef}
\kappa = 6.91 + \log s,
\end{equation}

\begin{equation}\label{tdef}
t = 7.61 + 2 \log p,
\end{equation}

\begin{equation}\label{chidef}
\chi = 6.91.
\end{equation}
Here $C_1, C_2, C_3$ are some positive constants independent of $N, \n, p,s, L$. In fact, probability $q$ can be made arbitrarily close to $1$ at the cost of worse constants.

It is worth noticing that, unusually, remainder terms $\Rb$, $\RBboot$ and $\Rsigma$ grow with $T$ defined by \eqref{Tdef} and hence with the sample size $N$ but the dependence is logarithmic. Indeed, we gain nothing from longer samples since we use only $2n$ data points each time.

\begin{lemma}\label{conditioning}
	Consider a measure $\mathbb{P}$ and a pair of sets $A$ and $B$. Then denoting $p \coloneqq \Prob{B}$
	\begin{equation}\label{key}
	\abs{\Prob{A} - \Prob{A | B}} \le 2(1-p).
	\end{equation}
\end{lemma}

\begin{proof}
	\begin{equation}
	\begin{split}
	\abs{\Prob{A} - \Prob{A | B}} &= \abs{\Prob{A | B} \Prob{B} + \Prob{A | \overline{B}} \Prob{\overline{B}} - \Prob{A| B}}\\
	& = \abs{\Prob{A | B} \left(p - 1\right) + \Prob{A | \overline{B}} (1-p)} \\
	& \le \abs{\Prob{A | B} \left(p - 1\right)} + \abs{\Prob{A | \overline{B}} (1-p)} \\
	& \le 2(1-p).
	\end{split}
	\end{equation}
\end{proof}

\section{Proof of the sensitivity result} \label{sensResProof}

\begin{theorem}\label{senstFin}
	Let Assumption \ref{subGaussianVector} hold. Also let $\deltaY < 1/2$ 	and
	\begin{equation}\label{smallresidual}
	\RBboot< \frac{\alpha}{6\abs{\n}},
	\end{equation}
	where $\deltaY$ and $\RBboot$ come from Lemmas \ref{trueHatSigma} and  \ref{SbGar}. Moreover, assume $\Istable \subseteq 1..\tau$ and $\tau \ge n_{suff}$, where
	\begin{equation} \label{nbound}
	n_{suff}\coloneqq \left(\frac{\q \infnorm{\inv{S}} - 2\rho  + \sqrt{\left(2\rho - \q\infnorm{\inv{S}}\right)^2 - 4\Delta \rho^2}}{\sqrt{2}\Delta}\right)^2,
	\end{equation}

	\begin{equation}\label{qdef}
	\q = \sqrt{2\left(1+\deltaY\right) \log \left(\frac{2  N\abs{\n} p^2}{\alpha -3\abs{\n}\RBboot}\right)},
	\end{equation}
	\begin{equation}\label{key}
	\rho = \sqrt{2\log p +\chi},
	\end{equation}
	where $\Delta$ denotes the break extent defined by \eqref{breakextentdef}.
	Let it hold for the widest window that $\nmax > n_{suff}$.
	Then with probability at least

	\begin{equation}
	\label{probsucc2}
	1 - \probzbound{s}{\kappa} - 3\probSigmaBound{s}{\chi}-\deltaOmegaProb{s}{t}{\x} - \probwbound{s}{\x}
	\end{equation}
	where $\probzbound{s}{\kappa}$, $\deltaOmegaProb{s}{t}{\x}$, $\probwbound{s}{\x}$ and $\probSigmaBound{s}{\chi}$,  come from Lemmas \ref{zboundlemmaAll}, \ref{omegaConcentation}, \ref{wboundlemma} and \ref{sigmaconcentration} respectively, the hypothesis $\Hnull$ will be rejected by the proposed approach at confidence level $\alpha$.

\end{theorem}

\paragraph{Discussion of the finite-sample-size sensitivity result}
The expression \eqref{nbound} and the residual $\RBboot$ involved in the statement of Theorem~\ref{senst} are rather complicated. Here we note that for $N$, $s$ and $p \rightarrow +\infty$, for some positive constant $C_4$ independent of $N$, $s$, $p$ and $\Delta$ it holds  that

\begin{equation}\label{nsuffbound}
n_{suff} \le C_4\left(1+\frac{\log^2(ps)}{\sqrt{s}}\right)\frac{\log \left(\abs{\n}Np^2\right)}{\Delta^2}
\end{equation}
while the bound \eqref{rbbootdisc} for $\RBboot$ holds as well, and the parameters $\x$, $t$ and $\kappa$ may be chosen as specified by \eqref{xdef}, \eqref{tdef} and \eqref{kappadef}  respectively and $\chi$ may be chosen as  $\chi=7.32$ in order to ensure the probability $\eqref{probsucc2}$ to be at least $0.99$.

As expected, the bound for sufficient window size decreases with the growth of the break extent $\Delta$ and the size of the set $\Istable$, but increases with dimensionality $p$. It is worth noticing, that the latter dependence is only logarithmic. And again, in the same way as with Theorem \ref{theTheoremFin}, the bound increases with the sample size $N$ (only logarithmically) since we use only $2n$ data points.

The assumption $\Istable \subseteq 1..\tau$ is only technical. The result may be proven without relying on it by methodologically the same argument.

Obviously, we still cannot explicitly compute $n_{suff}$, since it depends on the underlying distributions. However this result guarantees that the sensitivity of the test does not vanish in high-dimensional setting.

\begin{proof}[Proof of Theorem \ref{senstFin}]
	Consider a pair of centered normal vectors

	\begin{equation}
	\eta \coloneqq \left(\begin{array}{cccc}
		\eta^1 &\eta^2 &... &\eta^{\abs{\n}}
	\end{array}\right) \sim \N{0}{\trueSigmaY},
	\end{equation}
	\begin{equation}
		\zeta \coloneqq \left(\begin{array}{cccc}
			\zeta^1 &\zeta^2 &... &\zeta^{\abs{\n}}
		\end{array}\right) \sim \N{0}{\hatSigmaY},
	\end{equation}
	\begin{equation}
		\trueSigmaY \coloneqq\frac{1}{{2\nmax}} \sum_{j=1}^{2\nmax} \Var{Y^{n}_{\cdot j}},
	\end{equation}
	\begin{equation}
		\hatSigmaY \coloneqq\frac{1}{{2\nmax}} \sum_{j=1}^{2\nmax} \Var{Y^{n\flat}_{\cdot j}},
	\end{equation}
	where vectors $Y^n_{\cdot j}$ and $Y^{n\flat}_{\cdot j}$ are defined in proofs of Lemmas \ref{SGar} and \ref{SbGarRandomBounds} respectively.
	Lemma \ref{maxBound} applies here and yields for all positive $\q$
	\begin{equation}
	\Prob{ \infnorm{\zeta^{\nmax}} \ge \q } \le 2\abs{\Tnmax}p^2\exp\left(-\frac{\q^2 }{2\infnorm{\hatSigmaY} }\right),
	\end{equation}
	where $\hatSigmaY = \Var{\zeta}$ and $\abs{\Tnmax}$ is the number of central points for window of size $\nmax$.  Applying Lemma \ref{trueHatSigma} on a set of probability at least $1-\deltaOmegaProb{s}{t}{\x} - \probwbound{s}{\x} - \probSigmaBound{n}{\chi}$ yields
	$\infnorm{\trueSigmaY - \hatSigmaY} \le \deltaY$, and hence, due to the fact that $\infnorm{\trueSigmaY} = 1$ by construction,
	\begin{equation}
	\Prob{ \infnorm{\zeta^{\nmax}} \ge \q } \le 2\abs{\Tnmax} p^2\exp\left(-\frac{\q^2 }{2\left(1 + \deltaY\right)  }\right).
	\end{equation}
	Due to Lemma \ref{SbGar} and continuity of Gaussian c.d.f.
	\begin{equation}
	\Probboot{B^\flat_{\nmax} \ge \xalphann{\nmax}} \ge \alpha/\abs{\n} -  2\RBboot
	\end{equation}
	and due to Lemma \ref{SbGar} along with the fact that $\abs{\Tnmax} < N$, choosing $\q$ as proposed by equation \eqref{qdef}
	we ensure that $\xalphann{\nmax} \le\q $.

	Now using Lemma \ref{sigmaconcentration} twice for $\hatSigma^l_n(\tau)$ and $\hatSigma^r_n(\tau)$ respectively we obtain that with probability at least $1-2\probSigmaBound{n}{\chi}$

	\begin{equation}
	\Bn \ge \sqrt{\frac{n}{2}}\infnorm{{S}}\left(\Delta - 2\delta_n(\chi)\right).
	\end{equation}
	Finally, we notice that due to definition \eqref{nbound} of $n_{suff}$  and since $\nmax > n_{suff}$ $$B_{\nmax} > \q$$ and therefore, $\Hnull$ will be rejected.
\end{proof}

\begin{lemma}
	Consider a centered random Gaussian vector $\xi\in \R^p$ with an arbitrary covariance matrix $\Sigma$.
	For any positive $\q$ it holds that
	\begin{equation}
	\Prob{\max_i \xi_i \ge \q } \le p\exp\left(-\frac{\q^2 }{2\infnorm{\Sigma} }\right).
	\end{equation}
\end{lemma}

\begin{proof}
	For $\eta \sim \N{0}{\sigma^2}$ it holds that
	\begin{equation}
		\Prob{\eta > \q} \le \exp \left(-\frac{\q^2 }{2\infnorm{\Sigma} }\right).
	\end{equation}
	And clearly,
	\begin{equation}
		\begin{split}
			\Prob{\max_i \xi_i \ge \q } &= \Prob{\exists 1\le i \le p :\xi_i \ge \q} \\ &\le \sum_{i=1}^p\Prob{\xi_i \ge q}
			\\ &\le p\exp\left(-\frac{\q^2 }{2\infnorm{\Sigma} }\right).
		\end{split}
	\end{equation}
\end{proof}

As a trivial corollary, one obtains
\begin{lemma}\label{maxBound}
	Consider a centered random Gaussian vector $\xi\in \R^p$ with an arbitrary covariance matrix $\Sigma$.
	For any positive $\q$ it holds that
	\begin{equation}
	\Prob{ \infnorm{\xi} \ge \q } \le 2p\exp\left(-\frac{\q^2 }{2\infnorm{\Sigma} }\right).
	\end{equation}
\end{lemma}

\section{Sandwiching lemma}
The following lemma is a generalization covering the case of non-continuous probability measures of Lemma 21 of \cite{buzun}.

\begin{lemma}\label{sandwitch}
	Consider a normal multivariate vector $\eta$ with a deterministic covariance matrix and  a normal multivariate vector $\zeta$ with a possibly random covariance matrix such that
	\begin{equation}\label{gar}
	\sup_{\{x_n\}_{n \in \n} \subset \R} \abs{\Prob{\forall n \in \n : \Bn \le x_n} - \Prob{\forall n \in \n  : \infnorm{\eta_n} \le x_n}}  \le \Rb ,
	\end{equation}

	\begin{equation} \label{garb}
	\sup_{\{x_n\}_{n \in \n} \subset \R} \abs{\Probboot{\forall n \in \n : \Bbn \le x_n} - \Probboot{\forall n \in \n  : \infnorm{\zeta_{n}} \le x_n}}  \le \RBboot,
	\end{equation}

	\begin{equation} \label{tv}
	\sup_{\{x_n\}_{n \in \n} \subset \R} \abs{\Prob{\forall n \in \n : \Bn \le x_n} - \Probboot{\forall n \in \n : \Bbn \le x_n}} \le R.
	\end{equation}
	where $\eta_n$ and $\zeta_n$ are sub-vectors of $\eta$ and $\zeta$ respectively.
	Then

	\begin{equation}
	\abs{\Prob{\forall n \in \n : \Bn \le \xalphan} - (1-\alpha)} \le \left(3+2\abs{\n}\right) \left(R +\Rb + \RBboot\right).
	\end{equation}
\end{lemma}

\begin{proof}
	Let us introduce some notation. Denote multivariate cumulative distribution function of $\Bn, \Bbn, \infnorm{\eta_{n}}, \infnorm{\zeta_{n}}$ as   $P, \Pb, \Nn, \Nb : \R^{\abs{\n}} \rightarrow [0,1]$ respectively.
	Define sets for all $\delta\in [0,\alpha]$

	\begin{equation} \label{zmdef}
	\Zp(\delta) \coloneqq  \left\{z : \Nn(z) \ge 1-\alpha -\delta\right\},
	\end{equation}

	\begin{equation}
	\Zm(\delta) \coloneqq \left\{z : \Nn(z) \le  1-\alpha + \delta\right\}
	\end{equation}
	and their boundaries

	\begin{equation} \label{bzmdef}
	\bZp(\delta) \coloneqq  \left\{z : \Nn(z) = 1-\alpha - \delta \right\},
	\end{equation}

	\begin{equation}
	\bZm(\delta) \coloneqq \left\{z : \Nn(z) =  1-\alpha + \delta \right\}.
	\end{equation}
	Consider $\delta = R + \Rb + \RBboot $ and sets $\Zp = \Zp(\delta)\text{, } \Zm = \Zm(\delta)\text{, }  \bZm = \bZm(\delta)\text{, } \bZp = \bZp(\delta)$
	Define  a set of thresholds satisfying confidence level
	\begin{equation}
	\Zcb \coloneqq \left\{z : \Pb(z) \ge 1-\alpha ~\&~\forall z_1 < z :  \Pb(z_1) < 1-\alpha\right\}
	\end{equation}
	here and below comparison of vectors should be understood element-wise.
	Notice that due to continuity of multivariate normal distribution $\forall \zb \in \Zcb$  and assumption \eqref{garb}
	\begin{equation}\label{zbbound}
	\abs{\Pb(\zb) - (1-\alpha)}  \le \RBboot.
	\end{equation}
	Now for all $\zm \in \bZm$ and for all $\zb \in \Zcb$ it holds that

	\begin{equation}
	\begin{split}
		\Pb(\zm) &\le P(\zm) +R \\ &\le N(\zm) +R+\Rb \\ &\le 1-\alpha - \RBboot \\ &\le \Pb(\zb)
	\end{split}
	\end{equation}
	where we have consequently used \eqref{tv}, \eqref{gar}, \eqref{bzmdef} and \eqref{zbbound}.
	In the same way one obtains for all $\zp \in \bZp$ and for all $\zb \in \Zcb$
	\begin{equation}
	\Pb(\zp) \ge   \Pb(\zb)
	\end{equation}
	which implies that $\Zcb \subset \Zm \cap \Zp $.

	Now denote quantile functions of $\infnorm{\eta_n}$ as $z^N : [0,1] \rightarrow \R^{\abs{\n}} $:
	\begin{equation}
	\forall n \in \n : \Prob{\infnorm{\eta_n} \ge  z^N_n(\x) } = \x.
	\end{equation}
	In exactly the same way define quantile functions $\znb :[0,1] \rightarrow \R^{\abs{\n}} $ of $\infnorm{\zeta_n}$.
	Clearly for all $\x \in [0,1]$,
	\begin{equation}
	z^N(\x + \delta) \le  \zb(\x) \le z^N(\x - \delta)
	\end{equation}
	and hence
	\begin{equation}
	\zb(\alphacorrected) \le  z^N(\alphacorrected - \delta) \le \zb(\alphacorrected - 2\delta),
	\end{equation}
	\begin{equation}
	1-\alpha \le \Pb{( z^N(\alphacorrected - \delta))} \le \Pb{(\zb(\alphacorrected - 2\delta))},
	\end{equation}
	where $\alphacorrected$ is defined by \eqref{alphastardef}.
	Using Taylor expansion with Lagrange remainder term we obtain for some $0 \le \kappa \le 2\delta$
	\begin{equation}
	\begin{split}
	\Nb\left(\zb(\alphacorrected - 2\delta)\right) &\le \Nb\left(\znb(\alphacorrected - 2\delta)\right) +\delta \\&= \Nb\left( \znb(\alphacorrected )\right) + \sum_{n \in \n} \partial_{z_n^\flat}\Nb(\znb(\alphacorrected )) \partial_\alpha \znb_n(\alpha^*)\kappa  +\delta\\
	&\le 1-\alpha + \sum_{n \in \n} \partial_{z_n^\flat}\Nb(\znb(\alphacorrected )) \partial_\alpha \znb_n(\alpha^*)\kappa + 3\delta.
	\end{split}
	\end{equation}
	Next successively using Lemma \ref{l21} and the fact that quantile function is an inverse function of c.d.f. we obtain
	\begin{equation}
\Nb\left(\zb(\alphacorrected - 2\delta)\right)	\le 1-\alpha +3\delta+ 2\delta \abs{\n}.
	\end{equation}
	and therefore
	\begin{equation}
	1-\alpha \le \Pb\left(\zb(\alphacorrected - 2\delta)\right) \le 1-\alpha + \delta \left(3+2\abs{\n}\right),
	\end{equation}
	\begin{equation}
	1-\alpha \le \Pb\left(z^N(\alphacorrected - \delta)\right) \le 1-\alpha + \delta \left(3+2\abs{\n}\right).
	\end{equation}
	In the same way one obtains
		\begin{equation}
		1-\alpha - \delta \left(3+2\abs{\n}\right) \le \Pb\left(z^N(\alphacorrected + \delta)\right) \le 1-\alpha.
		\end{equation}
	Next, by the argument used in the beginning of the proof we obtain
	\begin{equation}
	z^N(\alphacorrected + \delta), z^N(\alphacorrected - \delta) \in \Zm( \delta \left(3+2\abs{\n}\right))\cap \Zp \left(\delta \left(3+2\abs{\n}\right)\right).
	\end{equation}
	As a final ingredient, we need to choose deterministic $\alphap$ and $\alpham$ such that (which is possible due to continuity)
	\begin{equation}
	\Nn(z^N(\alpham + \delta))  = 1-\alpha-\delta \left(3+2\abs{\n}\right),
	\end{equation}
	\begin{equation}
	\Nn(z^N(\alphap - \delta))  = 1-\alpha+\delta \left(3+2\abs{\n}\right)
	\end{equation}
	so $\alpham \le \alphacorrected \le \alphap$ and hence by monotonicity
	\begin{equation}
	z^N(\alpham + \delta) \le z^N(\alphacorrected + \delta)\le \zb(\alphacorrected) \le z^N(\alphacorrected - \delta) \le z^N(\alphap - \delta)
	\end{equation}
	and finally

	\begin{equation}
	\begin{split}
	1-\alpha-\delta \left(3 +2\abs{\n}\right)&\le  P(z^N(\alpham + \delta))\\
	&\le P(\zb(\alphacorrected)) \\
	&\le P(z^N(\alphap - \delta)) \\
	&\le 1-\alpha+\delta \left(3 +2\abs{\n}\right).
	\end{split}
	\end{equation}
\end{proof}

\begin{lemma}\label{l21}
	Consider a random variable $\xi$ and an event $A$ defined on the same probability space. Let c.d.f. $\Prob{\xi \le x}$ and $\Prob{\xi \le x \& A}$ be differentiable. Then
	\begin{equation}
	\frac{\partial_x \Prob{\xi \le x \cap A} }{\partial_x\Prob{\xi \le x}} \le 1
	\end{equation}
\end{lemma}
\begin{proof}
	Really denoting the complement of set $A$ as $\overline{A}$ we obtain,
	\begin{equation}
	\begin{split}
		\frac{\partial_x \Prob{\xi \le x \cap A} }{\partial_x\Prob{\xi \le x}} & = \frac{\partial_x \Prob{\xi \le x \cap A} }{\partial_x\left(\Prob{\xi \le x \cap  A} +\Prob{\xi \le x \cap \overline{A}}\right) }\\
		 & = \frac{\partial_x \Prob{\xi \le x \cap A} }{\partial_x \Prob{\xi \le x \cap  A} +\partial_x  \Prob{\xi \le x \cap \overline{A}} }\\
 		 & = \frac{1}{1 +\frac{\partial_x  \Prob{\xi \le x \cap \overline{A}}}{\partial_x \Prob{\xi \le x \cap A}} }
	\end{split}
	\end{equation}
	Using the fact that derivative of c.d.f. is non-negative we finalize the proof.
\end{proof}

\section{Similarity of joint distributions of $\{\Bn\}_{n\in \n}$ and $\{\Bbn\}_{n\in \n}$}

\begin{lemma} \label{tvlemma}
	Let Assumption \ref{subGaussianVector} hold and $\deltaY < 1/2$ where $\deltaY$ comes from Lemma \ref{trueHatSigma}.  Also let $X_1, X_2,...X_N$ be i.i.d. Then for all positive $\x$, $t$ and $\chi$ on a set of probability at least $1 - \probzbound{s}{\kappa} - \deltaOmegaProb{s}{t}{\x} - \probwbound{s}{\x} - \probSigmaBound{n}{\chi}$ 
	
	\begin{equation}\label{eqtv}
		\sup_{\{x_n\}_{n \in \n} \subset \R} \abs{\Prob{\forall n \in \n : \Bn \le x_n} - \Probboot{\forall n \in \n  : \Bbn \le x_n}}  \le R
	\end{equation}
	where 
	
	\begin{equation}
	R\coloneqq \Rb + \RBboot + \Rsigma
	\end{equation}
	
	\begin{equation}
	\Rsigma \coloneqq C\deltaY^{1/3} \log^{2/3}\left(Tp^2\right)
	\end{equation}
	$\probzbound{s}{\kappa}$, $\deltaOmegaProb{s}{t}{\x}$, $\probwbound{s}{\x}$ and $\probSigmaBound{n}{\chi}$,  come from Lemmas \ref{zboundlemmaAll}, \ref{omegaConcentation}, \ref{wboundlemma} and \ref{sigmaconcentration} respectively, $\Rb$ and $\RBboot$ are defined in Lemmas \ref{SGar} and \ref{SbGar} respectively and $C$ is an independent constant.  
	
\end{lemma}

\begin{proof}
	Consider a pair of normal vectors $\eta$ and $\zeta$ 
	\begin{equation}
	\eta \coloneqq \left(\begin{array}{cccc}
	\eta^1 &\eta^2 &... &\eta^{\abs{\n}} 
	\end{array}\right) \sim \N{0}{\trueSigmaY},
	\end{equation}
	\begin{equation}
	\zeta \coloneqq \left(\begin{array}{cccc}
		\zeta^1 &\zeta^2 &... &\zeta^{\abs{\n}} 
	\end{array}\right) \sim \N{0}{\hatSigmaY},
	\end{equation}
	\begin{equation}
	\trueSigmaY \coloneqq\frac{1}{{2\nmax}} \sum_{j=1}^{2\nmax} \Var{Y^{n}_{\cdot j}},
	\end{equation}	
	\begin{equation}
	\hatSigmaY \coloneqq\frac{1}{{2\nmax}} \sum_{j=1}^{2\nmax} \Var{Y^{n\flat}_{\cdot j}},
	\end{equation}	 
	where vectors $Y$ and $Y^\flat$ are defined in proofs of Lemmas \ref{SGar} and \ref{SbGar} respectively.  Applying Lemma \ref{gaussianComparison} along with Lemma \ref{twomatricies} yields 
	
	\begin{equation}
		\sup_{A \in A^{re}} \abs{\Prob{\eta \in A} - \Prob{\zeta \in A}} \le C\deltaY^{1/3} \log^{2/3}\left(Tp^2\right)
	\end{equation} 
	and the fact that $\forall k \in 1..p : \left(\Var{\zeta}\right)_{kk} = 1$ provides independence of the constant $C$. Here $A^{re}$ denotes a set of hyperrectangles in the sense of Definition \ref{hyperrectdef} and clearly  for all $\{x_n\}_{n \in \n} \subset \R$ the set 
$\left\{\forall n \in \n : \Bn < x_n \right\} $ is a hyperrectangle.
	Subsequently applying Lemmas \ref{SGar} and \ref{SbGar} we finalize the proof.
\end{proof}

\section{Gaussian approximation result for $\Bn$} \label{secsgar}

\begin{lemma}\label{SGar}
	Let Assumption \ref{subGaussianVector} hold. Then

	\begin{equation}
	\sup_{\{x_n\}_{n \in \n} \subset \R} \abs{\Prob{\forall n \in \n : \Bn \le x_n} - \Prob{\forall n \in \n  : \infnorm{\eta^n} \le x_n}}  \le \Rb
	\end{equation}
	Where
	\begin{equation}
	\left(\begin{array}{cccc}
	\eta^1 &\eta^2 &... &\eta^{\abs{\n}}
	\end{array}\right) \sim \N{0}{\trueSigmaY},
	\end{equation}
	\begin{equation}
	\trueSigmaY \coloneqq\frac{1}{{2\nmax}} \sum_{j=1}^{2\nmax} \Var{Y^n_{\cdot j}},
	\end{equation}

	\begin{equation}\label{key}
	\Rb \coloneqq 	\CB
	\left( F\log^7(2p^2T \nmax) \right)^{1/6},
	\end{equation}
	\begin{equation}\label{defF}
	F \coloneqq \frac{1}{2\nmin}\left(\beta \log 2 \vee \frac{\sqrt{2}}{\sqrt{2} - 1} \gamma\right)^2 \vee \frac{1}{2\nmax}\left( {\frac{\nmax}{\nmin}}\right)^{1/3} (\infnorm{\inv{S}}M_3)^2 \vee    \sqrt{\frac{1}{2\nmax\nmin}}   (\infnorm{\inv{S}} M_4)^2
	\end{equation}
	with $\gamma$ defined by \eqref{gammadef}, $\beta$ by \eqref{betadef} and $Y$ along with its sub-matrices $Y^n$ by \eqref{BFMTogether} and \eqref{BFM}. Also, $M_3^3$ and $M_4^4$ stand for the third and the fourth maximal centered moments of products $(X_1)_k(X_1)_l$ of pairwise components of $X_1$ and $\CB$ is an independent constant.

\end{lemma}

\begin{proof}
	First, we define for all $i \in 1..n$

	\begin{equation}\label{Zrealdef}
	Z_i \coloneqq \inv{S}\left(\overline{X_i X_i^T - \trueSigma}\right)
	\end{equation}
	and notice that
	\begin{equation} \label{bootAnDef}
	\Bn(t) \coloneqq \infnorm{\frac{1}{\sqrt{2n}} \left( \sum_{i \in \Il} \Zi  - \sum_{i \in \Ir} \Zi \right)}.
	\end{equation}
	Next, consider a matrix $Y_n$ with $2\nmax$ columns

		\begin{equation} \label{BFM}
		\begin{split}
		(Y^n)^T&\coloneqq \sqrt{\frac{\nmax}{n}} \times \\
&		\left( \begin{array}{cccccc}
	Z_1 & O & ... & O & -Z_{2\nmax+1} & ... \\
	Z_2 & Z_2 & ... & ... & ... & ... \\
	... & Z_3 & ... & ... & ... & ... \\
	Z_n & ... & ... & ... & ... & ... \\
	-Z_{n+1} & Z_{n+1} & ... & ... & ... & ... \\
	-Z_{n+2} & -Z_{n+2} & ... & ... & ... & ... \\
	... & -Z_{n+3} & ... & O & ... & ... \\
	-Z_{2n} & ... & ... & Z_{2\nmax-2n+1} & O & ... \\
	O & -Z_{2n+1} & ... & Z_{2\nmax-2n+2} & Z_{2\nmax-2n+2} & ... \\
	O & O & ... & ... & ... & ... \\
	... & ... & ... & -Z_{2\nmax - 1} & -Z_{2\nmax - 1} & ... \\
	O & O & ... & -Z_{2\nmax} & -Z_{2\nmax} & ...  \\\end{array} \right).
		\end{split}
	\end{equation}
	Clearly, columns of the matrix are independent and
	\begin{equation}
	\Bn = \frac{1}{\sqrt{2\nmax}} \sum_{l=0}^{2\nmax} (Y^n)_{\cdot l}
	\end{equation}
	Next, we define a block matrix composed of $Y_n$ matrices:

	\begin{equation} \label{BFMTogether}
	Y \coloneqq \left(\begin{array}{c}
	Y^1 \\ \hline
	Y^2 \\ \hline
	... \\ \hline
	Y^{\numn}
	\end{array}\right).
	\end{equation}
	Again, vectors $Y_{\cdot l}$ are independent and for all $\{x_n\}_{n \in \n} \subset \R$ the set
	\begin{equation}
	\left\{\forall n \in \n : \Bn \le x_n \right\}
	\end{equation}
	is a hyperrectangle in the sense of Definition \ref{hyperrectdef}.

	The rest of the proof consists in applying Lemma \ref{generalGAR}.
	Denote

	\begin{equation}\label{Bndef}
	G_{\nmax} = \sqrt{\frac{\nmax}{ \nmin}}\left(\beta \log 2 \vee \frac{\sqrt{2}}{\sqrt{2} - 1} \gamma\right) \vee \left( {\frac{\nmax}{ \nmin}}\right)^{1/6} M_3 \vee  \left({\frac{\nmax}{ \nmin}}\right)^{1/4} M_4.
	\end{equation}
	In the same way as in Lemma \ref{zboundlemma} one shows that the assumptions of Lemma \ref{expmoment} hold for components of $Z_i$ with

	\begin{equation} \label{gammadef}
	\gamma \coloneqq L^2 \infnorm{\inv{S}},
	\end{equation}

	\begin{equation} \label{betadef}
	\beta \coloneqq  L^2 \infnorm{\inv{S}} \infnorm{\trueSigma}.
	\end{equation}
	Therefore, condition \eqref{garassumption3} holds with $G_{\nmax}$ defined by equation \eqref{Bndef}.
	In order to see that condition \eqref{garass1} is fulfilled with $b = 1$ notice that
	\begin{equation}
	\frac{1}{2\nmax} \sum_{j=1}^{\nmax} \E{(Y_{ij}^{n})^2} \ge  \min_j \Var{(Z_1)_j}  = 1.
	\end{equation}
	Next, observe that for any $k$-th component $Z_{ik}$ of $Z_i$ and a central point $t$ (both determined by $j$):

	\begin{equation}
	\begin{split}
	\frac{1}{2\nmax} \sum_{j=1}^{2\nmax} \E{\abs{Y_{ij}^n}^3} & =  \frac{1}{2\nmax} \sum_{i \in \Il \cup\Ir }  \E{\left(\sqrt{\frac{\nmax}{n}} \abs{Z_{ik}}\right)^3 } \\
	& = \frac{1}{2\nmax} \sum_{i \in \Il \cup\Ir } \left(\frac{\nmax}{n}\right)^{3/2} \E{\abs{Z_{ik}}^3} \\
	& = \frac{2n}{2\nmax} \left(\frac{\nmax}{n}\right)^{3/2} \E{\abs{Z_{ik}}^3} \\
	& = \sqrt{\frac{\nmax}{n}} \E{\abs{Z_{ik}}^3}\\
	&\le \sqrt{\frac{\nmax}{\nmin}} \left(\infnorm{\inv{S}}M_3\right)^3.
	\end{split}
	\end{equation}
	In the same way:

	\begin{equation}
	\frac{1}{2\nmax} \sum_{i=1}^{N} \E{\abs{Y_{ij}^n}^4} \le \frac{\nmax}{ \nmin} \left(\infnorm{\inv{S}} M_4\right)^4.
	\end{equation}
	Therefore, condition \eqref{garassumptions} holds with $B_{\nmax}$,
	so Lemma \ref{generalGAR} applies here and provides us with the claimed bound. Moreover,
	$\CB$ depends only on $b=1$ which implies that the constant $\CB$ is independent.
\end{proof}

\begin{lemma} \label{expmoment}
	Consider a random variable $\xi$. Suppose $\forall \x \ge 0$ the following bound holds:

	\begin{equation}
	\Prob{\abs{\xi}\ge \gamma \x + \beta}  \le \ex{-\x}.
	\end{equation}
	Then

	\begin{equation}
	\E{\exp\left( \frac{\abs{\xi}}{G} \right)} \le 2
	\end{equation}
	for

	\begin{equation}
	G = \beta \log 2 \vee \frac{\sqrt{2}}{\sqrt{2} - 1} \gamma.
	\end{equation}

\end{lemma}

\begin{proof}

	Integration by parts yields

	\begin{equation}
	\E{\exp\left( \frac{\abs{\xi}}{G} \right)} \le \exp\left(\frac{\beta}{G}\right) + \frac{\gamma}{G} \int_{0}^{+\infty} \exp \left(\frac{\gamma x + \beta}{G}\right) \ex{-x} dx.
	\end{equation}

	\begin{equation}
	\int_{0}^{+\infty} \exp \left(\frac{\gamma x + \beta}{G}\right) \ex{-x} dx
	= \frac{G}{G - \gamma } \exp\left(\frac{\beta}{G}\right).
	\end{equation}

	\begin{equation}\begin{split}
	\E{\exp\left( \frac{\abs{\xi}}{G} \right)} & \le \frac{G}{G - \gamma } \exp\left(\frac{\beta}{G}\right) \\
	& \le 2.
	\end{split}
	\end{equation}
\end{proof}
Using the same technique the following lemma, which may be of use in order to bound the moments $M^3_3$ and $M^4_4$, can be proven.

\begin{lemma}\label{thirdAndFourthMomentBounds}
	Under assumptions of Lemma \ref{expmoment}

	\begin{equation}
	\E{\abs{\xi}^3} \le \beta^3 + 3\gamma\beta^2 + 6 \beta\gamma^2+2\gamma^3,
	\end{equation}

	\begin{equation}
	\E{\xi^4} \le \beta^4 + 4\gamma\beta^3 + 12 \beta^2 \gamma^2 6\beta\gamma^3 + 24 \gamma^4.
	\end{equation}

\end{lemma}

\begin{lemma} \label{zboundlemma}
	Under Assumption \ref{subGaussianVector} it holds for all  $i \in 1..N$ and positive $\kappa$ that

	\begin{equation}
	\Prob{\forall k \in 1..p : \abs{(\Zi)_k} \le \infnorm{\inv{S}} L^2 \left(\kappa + \log p +\infnorm{\trueSigma}\right) } \ge 1 - \ex{-\kappa}.
	\end{equation}
\end{lemma}

\begin{proof}
	According to the definition \eqref{Zrealdef} of $\Zi$ for its  arbitrary element $(\Zi)_k$ one obtains for some $l,m \in 1..p$:
	\begin{equation}
	(\Zi)_k = \inv{S}_{kk} \left((X_i)_l (X_i)_m - \trueSigma_{lm} \right).
	\end{equation}
	By sub-Gaussianity Assumption \ref{subGaussianVector} it holds for all positive $x$ that

	\begin{equation}\label{xbound}
	\Prob{\forall k \in 1..p : \abs{(X_i)_k} \le x } \ge 1-p\ex{-x^2/L^2}.
	\end{equation}
	Hence
	\begin{equation}
	\Prob{\forall k \in 1..p : \abs{(\Zi)_k} \le \infnorm{\inv{S}}\left(x^2 + \infnorm{\trueSigma}\right) } \ge 1-p\ex{-x^2/L^2}
	\end{equation}
	and finally a change of variables establishes the claim.

\end{proof}

\section{Gaussian approximation result for $\Bbn$} \label{secsbgar}

Denote

\begin{equation}\label{truesigmaYDef}
\trueSigmaY \coloneqq \frac{1}{2\nmax} \sum_{i=1}^{2\nmax} \Var{Y_{\cdot i}}
\end{equation}

\begin{equation}\label{hatsigmaYDef}
\hatSigmaY \coloneqq \frac{1}{2\nmax} \sum_{i=1}^{2\nmax} \Var{Y_{\cdot i}^\flat}
\end{equation}
where vectors $Y_{\cdot j}$ and $Y^\flat_{\cdot j}$ are defined by \eqref{BFM} and  \eqref{BFMbootstrapTogether} respectively.

\begin{lemma}\label{SbGarRandomBounds}
	\begin{equation}
	\sup_{\{x_n\}_{n \in \n} \subset \R} \abs{\Probboot{\forall n \in \n : \Bbn \le x_n} - \Probboot{\forall n \in \n  : \infnorm{\zeta^n} \le x_n}}  \le \hatCBb
	\left( F^\flat\log^7(2p^2T \nmax) \right)^{1/6}
	\end{equation}
	Where $\hatCBb$ depends only on $\min_{k \in 1..p} (\hatSigmaY)_{kk}$,
	\begin{equation}
	\left(\begin{array}{cccc}
	\zeta^1 &\zeta^2 &... &\zeta^{\abs{\n}}
	\end{array}\right) \sim \N{0}{\hatSigmaY},
	\end{equation}
	\begin{equation}
	\hatSigmaY \coloneqq\frac{1}{{2\nmax}} \sum_{j=1}^{2\nmax} \Var{Y^{n\flat}_{\cdot j}},
	\end{equation}
	\begin{equation}
	F^\flat = \left(\frac{1}{2\nmin\log^22}  \vee \frac{1}{2\nmax}\left({\frac{\nmax}{\nmin}}\right)^{1/3} \vee  \sqrt{\frac{1}{2\nmax\nmin}}  \right)\infnorm{\inv{S}}^2 (M^\flat)^2
	\end{equation}

	\begin{equation}
	M^\flat = \max_{i \in \Istable} \infnorm{\hat Z_i} .
	\end{equation}
\end{lemma}

\begin{proof}
	This proof is similar to the proof of Lemma \ref{SGar}.

	Consider a matrix which is a bootstrap counterpart of $Y^n$
	\begin{equation}   \label{BFMbootstrap}
	\begin{split}
	(Y^{n\flat})^T&\coloneqq \sqrt{\frac{\nmax}{n}} \times\\
	&\left( \begin{array}{cccccc}
	\Zb_1 & O & ... & O & -\Zb_{2\nmax+1} & ... \\
	\Zb_2 & \Zb_2 & ... & ... & ... & ... \\
	... & \Zb_3 & ... & ... & ... & ... \\
	\Zb_n & ... & ... & ... & ... & ... \\
	-\Zb_{n+1} & \Zb_{n+1} & ... & ... & ... & ... \\
	-\Zb_{n+2} & -\Zb_{n+2} & ... & ... & ... & ... \\
	... & -\Zb_{n+3} & ... & O & ... & ... \\
	-\Zb_{2n} & ... & ... & \Zb_{2\nmax-2n+1} & O & ... \\
	O & -\Zb_{2n+1} & ... & \Zb_{2\nmax-2n+2} & \Zb_{2\nmax-2n+2} & ... \\
	O & O & ... & ... & ... & ... \\
	... & ... & ... & -\Zb_{2\nmax - 1} & -\Zb_{2\nmax - 1} & ... \\
	O & O & ... & -\Zb_{2\nmax} & -\Zb_{2\nmax} & ... \\
	\end{array} \right).
	\end{split}
	\end{equation}
	Clearly, columns of the matrix are independent and
	\begin{equation}
	\Bbn = \frac{1}{\sqrt{2\nmax}} \sum_{l=0}^{2\nmax} (Y^{ n\flat})_{\cdot l}
	\end{equation}
	Next, we define a block matrix composed of $Y^{n\flat}$ matrices:

	\begin{equation} \label{BFMbootstrapTogether}
	\Yb \coloneqq \left(\begin{array}{c}
	Y^{ 1\flat} \\ \hline
	Y^{2\flat } \\ \hline
	... \\ \hline
	Y^{\numn\flat }
	\end{array}\right).
	\end{equation}
	Again, vectors $Y_{\cdot l}^\flat$ are independent and for all $\{x_n\}_{n \in \n} \subset \R$ the set
	\begin{equation}
	\left\{\forall n \in \n : \Bbn < x_n \right\}
	\end{equation}
	is a hyperrectangle in the sense of Definition \ref{hyperrectdef}.
	Now notice that
	\begin{equation}
	\frac{1}{2\nmax} \sum_{j=1}^{2\nmax} \E{\abs{Y_{ij}}^3} \le \sqrt{\frac{\nmax}{\nmin}}\infnorm{\inv{S}}^3 \max_{i \in \Istable} \infnorm{Z_i},
	\end{equation}

	\begin{equation}
	\frac{1}{2\nmax} \sum_{j=1}^{2\nmax} \E{\abs{Y_{ij}}^4} \le \frac{\nmax}{\nmin} \infnorm{\inv{S}}^4 \max_{i \in \Istable} \infnorm{ Z_i}.
	\end{equation}
	And finally apply Lemma \ref{generalGAR}.
\end{proof}

\begin{lemma} \label{zboundlemmaAll}
	Under Assumption \ref{subGaussianVector} it holds for all positive $\kappa$ that

	\begin{equation}
	\Prob{\forall i \in \Istable: \infnorm{Z_i}\le \Zbound{s}{\kappa}  } \ge 1 - \probzbound{s}{\kappa}
	\end{equation}
	where

	\begin{equation}
	\Zbound{s}{\kappa} \coloneqq \infnorm{\inv{S}} L^2 \left(\kappa + \log p +\infnorm{\trueSigma}\right),
	\end{equation}

	\begin{equation}
	\probzbound{s}{\kappa} \coloneqq  s\ex{-\kappa}.
	\end{equation}

\end{lemma}

\begin{proof}
	Proof of the Lemma consists in applying Lemma \ref{zboundlemma} and appropriate multiplicity correction.
\end{proof}

\begin{lemma}\label{SbGar}
	Let Assumption \ref{subGaussianVector} hold and $\deltaY < 1/2$. Then for all positive $\kappa$ with probability at least $1- \probzbound{s}{\kappa}$

	\begin{equation}
	\sup_{\{x_n\}_{n \in \n} \subset \R} \abs{\Probboot{\forall n \in \n : \Bbn \le x_n} - \Probboot{\forall n \in \n  : \infnorm{Y^{n\flat}} \le x_n}}  \le \RBboot
	\end{equation}
	Where

	\begin{equation}\label{key}
	\RBboot \coloneqq \CBb
	\left( \hat F\log^7(2p^2T \nmax) \right)^{1/6},
	\end{equation}

	\begin{equation}
	\hat F \coloneqq \left(\frac{1}{2\nmin\log^22}  \vee \frac{1}{2\nmax}\left({\frac{\nmax}{\nmin}}\right)^{1/3} \vee  \sqrt{\frac{1}{2\nmax\nmin}}  \right)\infnorm{\inv{S}}^2 (\Zbound{s}{\kappa})^2
	\end{equation}
	and $\hatCBb$ is an independent constant.
\end{lemma}

\begin{proof}
	The proof consists in subsequently applying Lemmas \ref{SbGarRandomBounds} and \ref{zboundlemmaAll} ensuring $M^\flat \le \Zbound{s}{\kappa}$ with probability at least $ 1- \probzbound{s}{\kappa}$, while assumed bound $\infnorm{\trueSigmaY - \hatSigmaY} \le \deltaY < 1/2 = \min_{1 \le k \le p} (\trueSigmaY)_{kk}$ implies the existence of a deterministic constant $\CBb > \hatCBb$.
\end{proof}

\section{$\trueSigmaY \approx \hatSigmaY$} \label{sec2mat}

Denote

\begin{equation}
W_i \coloneqq \overline{X_i X_i^T},
\end{equation}

\begin{equation}
\trueOmega \coloneqq  \E{\left(W_1 - \overline{\trueSigma}\right) \left(W_1 - \overline{\trueSigma}\right)^T},
\end{equation}

\begin{equation}
\hatOmega \coloneqq\empE{{\left(W_i - \overline{\trueSigma}\right) \left(W_i - \overline{\trueSigma}\right)^T}}
\end{equation}
where notation $\empE{\cdot}$ is used as a shorthand for averaging over $\Istable$, e.g.

\begin{equation}\label{key}
\empE{\xi_i} = \frac{1}{s} \sum_{i\in \Istable} \xi_i,
\end{equation}
and similarly $\empVar{\cdot}$ denotes an empirical covariance matrix computed using the same set, e.g.
\begin{equation}\label{key}
\empVar{\xi_i} = \empE{\xi_i \xi_i^T}.
\end{equation}

The results of this section rely on the following lemma which is a trivial corollary of Lemma 6 by \cite{sara} providing the concentration result for empirical covariance matrix.

\begin{lemma}
	\label{sigmaconcentration}
	Consider an i.i.d. $p$-dimensional sample of length $s$. Let Assumption \ref{subGaussianVector} hold for some $L > 0$. Then for any positive $\chi$
	\begin{equation}
	\Prob{\infnorm{\overline{\trueSigma}  - \empE{W_i}} \ge \delta_s(\chi) } \le \probSigmaBound{s}{\chi} \coloneqq 2\ex{-\chi},
	\end{equation}
	where

	\begin{equation}
	\delta_s(\chi) \coloneqq 2L^2 \left(\frac{2 \log p + \chi}{s} + \sqrt{\frac{4\log p + 2\chi}{s}}\right).
	\end{equation}

\end{lemma}

Straightforwardly applying Assumption \ref{subGaussianVector} and a proper multiplicity correction yields the following result.
\begin{lemma} \label{wboundlemma}
	Under Assumption \ref{subGaussianVector} it holds for all  positive $\x$ that

	\begin{equation}
	\Prob{\forall i \in \Istable : \infnorm{W_i - \overline{\trueSigma}} \le \Wbound{s}{\x}  } \ge 1 - \probwbound{s}{\x},
	\end{equation}
	where

	\begin{equation}
	\Wbound{s}{\x} \coloneqq  \x^2 + \infnorm{\trueSigma},
	\end{equation}

	\begin{equation}
	\probwbound{s}{\x}  \coloneqq ps\ex{-\x}.
	\end{equation}

\end{lemma}

\begin{lemma} \label{bootstrapmistie}
	Under Assumption \ref{subGaussianVector} with probability at least $1 - \probwbound{\x}{s} - \probSigmaBound{n}{\chi}$

	\begin{equation}
		\infnorm{\empVar{W_i - \empE{W_i}} - \hatOmega} \le 2\Wbound{s}{\x} \delta_s(\chi) +\delta_s(\chi)^2.
	\end{equation}

\end{lemma}

\begin{proof}
	By the construction of bootstrap procedure	and definition \eqref{Zrealdef}
	\begin{equation}
	\begin{split}
		\empVar{W_i - \empE{W_i}} &= \frac{1}{s}\sum_{i \in \Istable} \left(W_i - \empE{W_i}\right)\left(W_i - \empE{W_i}\right)^T \\
		&= \frac{1}{s}\sum_{i \in \Istable} \left(W_i- \overline{\trueSigma} + \overline{\trueSigma} - \empE{W_i}\right)\left(W_i - \overline{\trueSigma} + \overline{\trueSigma} - \empE{W_i}\right)^T \\
		&= \hatOmega + \frac{1}{s}\sum_{i \in \Istable} \left(\overline{\trueSigma}  - \empE{W_i}\right)\left(\overline{\trueSigma}  - \empE{W_i}\right)^T + \frac{2}{s}\left(W_i - \overline{\trueSigma}\right)\left(\overline{\trueSigma}  - \empE{W_i}\right)^T.
	\end{split}
	\end{equation}
	Applying Lemmas \ref{wboundlemma} and \ref{sigmaconcentration} yields the claim.
\end{proof}

\begin{lemma}\label{omegaConcentation}
	Let Assumption \ref{subGaussianVector} hold for some $L > 0$.  Then for any positive $t$ and $\x$

	\begin{equation}
	\Prob{\infnorm{\trueOmega - \hatOmega} \ge \deltaOmega{s}{t}{\x} } \le \deltaOmegaProb{s}{t}{\x}
	\end{equation}
	where
	\begin{equation}
	\deltaOmega{s}{t}{\x} \coloneqq \Bernstain{2(\Wbound{s}{\x})^2}{t}{s}{\sigma_\Omega^2},
	\end{equation}
	\begin{equation}
	\deltaOmegaProb{s}{t}{\x} \coloneqq p^2e^{-t} + \probwbound{s}{\x}.
	\end{equation}
\end{lemma}

\begin{proof}
	Consider a random variable

	\begin{equation}
	\zeta_{lm}^i \coloneqq  \left(W_i - \overline{\trueSigma}\right)_l \left(W_i - \overline{\trueSigma}\right)_m -  \trueOmega_{lm} .
	\end{equation}
	By Lemma \ref{wboundlemma} we can bound it as $\abs{\zeta_{lm}^i} \le 2(\Wbound{s}{\x})^2 $ with probability at least $1-\probwbound{s}{\x}$. Due to $\zeta_{ij}^i$ being centered Bernstein inequality applies:

	\begin{equation}
	\Prob{\empE{\zeta_{lm}^i} \ge \Bernstain{2(\Wbound{s}{\x})^2}{t}{s}{\sigma_\Omega^2} } \le e^{-t}.
	\end{equation}

\end{proof}

\begin{lemma} \label{trueHatSigma}
	Under Assumption \ref{subGaussianVector} for any positive $t$, $\x$ and $\chi$ with probability at least $ 1-\deltaOmegaProb{s}{t}{\x} - \probwbound{s}{\x} - \probSigmaBound{n}{\chi}$ it holds that

	\begin{equation}
		\infnorm{\Var{Z_i} - \Var{\Zbi}} \le \deltaY \coloneqq \infnorm{\inv{S}}^2 \left(	\deltaOmega{s}{t}{\x}  + 2\Wbound{s}{\x} \delta_s(\chi) +\delta_s^2 (\chi) \right).
	\end{equation}
\end{lemma}

\begin{proof}
	Proof consists in applying Lemmas \ref{omegaConcentation} and \ref{bootstrapmistie}.
\end{proof}

Using the fact that the covariance matrices $\trueSigmaY$ and $\hatSigmaY$ are block matrices composed of blocks $\Var{Z_i}$ and $\Var{\Zbi}$ respectively, multiplied by some positive values $\le 1$, we trivially obtain the following result.

\begin{lemma}\label{twomatricies}
	Under Assumption \ref{subGaussianVector} for any positive $t$, $\x$ and $\chi$ with probability at least $ 1-\deltaOmegaProb{s}{t}{\x} - \probwbound{s}{\x} - \probSigmaBound{n}{\chi}$ it holds that

	\begin{equation}
	\infnorm{\hatSigmaY- \trueSigmaY} \le \deltaY
	\end{equation}
	where $\deltaY$ comes from Lemma \ref{trueHatSigma}.
\end{lemma}

\section{General Gaussian approximation result} In this section we briefly describe the result obtained in \cite{Chernozhukov2014} .

Throughout this section consider an  independent sample $x_1, ... , x_n \in \R^p$ of centered random variables. Define their Gaussian counterparts  $y_i \sim \N{0}{\Var{x_i}}$ and denote their scaled sums as

	\begin{equation}
	S^X_n \coloneqq \frac{1}{\sqrt{n}} \sum_{i=1}^n x_i \text{ and }S^Y_n \coloneqq \frac{1}{\sqrt{n}} \sum_{i=1}^n y_i .
	\end{equation}

\begin{definition}\label{hyperrectdef}
We call a set $A$ of a form $A = \{w \in \R^p : a_i \le w_i \le b_i ~\forall i \in \{1..p\} \}$ a hyperrectangle. A family of all hyperrectangles is denoted as $A^{re}$.
\end{definition}

\begin{assumption} \label{garass1}
	$\exists b > 0$ such that

	\begin{equation} \label{garassumption1}
	\frac{1}{n} \sum_{i=1}^{n}\E{x_{ij}^2} \ge b \text{ for all } j \in \{1..p\}.
	\end{equation}

\end{assumption}

\begin{assumption} \label{garassumptions}
	$\exists G_n \ge 1$ such that

	\begin{equation} \label{garassumption2}
	\frac{1}{n} \sum_{i=1}^{n} \E{\abs{x_{ij}}^{2+k}} \le G_n^{2+k} \text{ for all } j \in \{1..p\} \text{ and } k \in \{1,2\},
	\end{equation}

	\begin{equation} \label{garassumption3}
	\E{\exp \left(\frac{\abs{x_{ij}}}{G_n}\right)} \le 2 \text{ for all } j \in \{1..p\} \text{ and } i \in \{1.. n\}.
	\end{equation}

\end{assumption}

\begin{lemma}[Proposition 2.1 by \cite{Chernozhukov2014}]\label{generalGAR}
	Let Assumption \ref{garass1} hold for some $b$ and Assumption  \ref{garassumptions} hold for some $G_n$.
	Then

	\begin{equation}
	\sup_{A \in A^{re}} \abs{\Prob{S^X_n \in A}  - \Prob{S^Y_n \in A}} \le C\left(\frac{G_n^2\log^7(pn)}{n}\right)^{1/6}
	\end{equation}
	and $C$ depends only on $b$.

\end{lemma}

\section{Gaussian comparison result}

By the technique given in the proof of Theorem 4.1 by \cite{Chernozhukov2014} one obtains the following generalization of the result given in \cite{anticonc}

\begin{lemma} \label{gaussianComparison}
	Consider a pair of covariance matrices $\Sigma_1$ and $\Sigma_2$ of size $p\times p$ such that
	\begin{equation}
	\infnorm{\Sigma_1 - \Sigma_2} \le \Delta
	\end{equation} 
	and $\forall k : C_1 \ge (\Sigma_1)_{kk} \ge c_1 > 0$.
	Then for random vectors $\eta \sim \N{0}{\Sigma_1}$ and $\zeta \sim \N{0}{\Sigma_2}$ it holds that 
	
	\begin{equation}
	\sup_{A \in A^{re}} \abs{\Prob{\eta \in A} - \Prob{\zeta \in A}} \le C\Delta^{1/3} \log^{2/3}p,
	\end{equation} 
	where $C$ is a positive constant which depends only on $C_1$ and $c_1$.
\end{lemma}

\end{document}